\documentclass[a4paper,11pt]{amsart}
\usepackage{amsfonts}
\usepackage{anysize} \marginsize{1.3in}{1.3in}{1in}{1in}
 \usepackage{enumerate}
\usepackage{amsmath, amsthm}
\usepackage{amscd}
\usepackage{amssymb}
\usepackage{pdfsync}
\usepackage[title]{appendix}
\usepackage{xy}
\xyoption{all}
\usepackage{latexsym,graphicx}
\usepackage[latin1]{inputenc}
\usepackage{xspace}
\usepackage{array}
\usepackage{enumitem,kantlipsum}
\usepackage{newclude}
\usepackage{hyperref}
\hypersetup{pdfpagelabels,
                  plainpages=false,
                    colorlinks=true,       
                    linkcolor=red}
\renewcommand*{\HyperDestNameFilter}[1]{\jobname-#1} 
\numberwithin{equation}{section}
\usepackage[centering]{geometry}
\usepackage{todonotes}

\usepackage[nameinlink]{cleveref}

\newcommand{\noi}{\noindent}

 \theoremstyle{plain}
\newtheorem{theor}{Theorem}[section]
\newtheorem{conj}[theor]{Conjecture}
\newtheorem{prop}[theor]{Proposition}
\newtheorem{lem}[theor]{Lemma}

\newtheorem{cor}[theor]{Corollary}

\theoremstyle{remark}
\newtheorem{rem}[theor]{Remark}

\theoremstyle{plain}
\newtheorem{defi}[theor]{Definition}

\numberwithin{equation}{section}

\newcommand{\CC}{{\mathbb C}}

\renewcommand{\SS}{{\mathbf S}}
\newcommand{\QQ}{{\mathbb Q}}

\newcommand{\ZZ}{{\mathbb Z}}
\newcommand{\VV}{{\mathbb V}}

\newcommand{\G}{{\mathbf G}}
\newcommand{\HH}{{\mathbf H}}

\newcommand{\NN}{{\mathbb N}}

\newcommand{\Spec}{{\rm Spec}\,}

\newcommand{\HL}{\textnormal{HL}}

\newcommand{\lo}{\longrightarrow}

\newcommand{\Hom}{{\rm Hom}}

\newcommand{\Gal}{{\rm Gal}}
\newcommand{\ad}{{\rm ad}}

\newcommand{\der}{{\rm der}}
\newcommand{\nor}{{\rm nor}}

\newcommand{\GL}{{\rm \bf GL}}

\newcommand{\F}{\mathcal{F}}

\newcommand{\Aut}{\textnormal{Aut}}
\newcommand{\an}{\textnormal{an}}

\newcommand{\cT}{{\mathcal T}}

\newcommand{\cA}{{\mathcal A}}

\newcommand{\cD}{{\mathcal D}}

\newcommand{\cV}{{\mathcal V}}
\newcommand{\cO}{{\mathcal O}}
\newcommand{\cY}{{\mathcal Y}}

\newcommand{\oQ}{{\overline{\QQ}}}

\newcommand{\Hod}{\textnormal{Hod}}

\newcommand{\Hdg}{\textnormal{Hdg}}
\newcommand{\ws}{\textnormal{ws}}
\newcommand{\norm}{\textnormal{nor}}

\newcommand{\MIC}{\textnormal{MIC}}
\newcommand{\Loc}{\textnormal{Loc}}

\newcommand{\prim}{\textnormal{prim}}

\begin{document}
\title{On the fields of definition of Hodge loci}
\author{B. Klingler, A. Otwinowska and D. Urbanik}

\begin{abstract}
  Given a polarizable variation of $\ZZ$-Hodge structure $\VV$ over a
  smooth quasi-projective complex variety $S$ Cattani, Deligne and
  Kaplan proved that the Hodge locus of closed points $s \in S$ such that $\VV_s$ admits exceptional
Hodge tensors is a countable union of strict closed irreducible algebraic
subvarieties of $S$, called the special subvarieties of $S$ for
$\VV$.

When $\VV$ is moreover defined over a number field $L \subset \CC$
i.e. both $S$ and the filtered algebraic module 
with integrable connection $(\cV, F^\bullet, \nabla)$ associated with $\VV$ are
defined over $L$, any special subvariety of $S$ for $\VV$ is
conjectured to be defined over $\oQ$, and
its $\Gal(\oQ/L)$-conjugates to be again special subvarieties for
$\VV$. In the geometric case this follows from the conjecture that Hodge classes are absolute Hodge.

We prove that if $S$ is defined over a number field $L$ then any special
subvariety of $S$ for $\VV$ which is weakly
non-factor is defined over $\oQ$; and that its
$\Gal(\oQ/L)$-conjugates are special if moreover $\VV$ is
defined over $L$.  The non-factor condition
roughly means that the special subvariety cannot be non-trivially 
Hodge-theoretically deformed inside a larger special subvariety.

Our result implies that if $S$ is defined over a number field $L \subset \CC$ and if the adjoint group
of the generic Mumford-Tate group of $\VV$ is simple then any strict special
subvariety of $S$ for $\VV$ with non-trivial algebraic monodromy and which
is maximal for these properties is defined over $\oQ$; and that its
$\Gal(\oQ/L)$-conjugates are special if moreover $\VV$ is defined over
$L$. It also implies   
that special subvarieties for $\ZZ$VHSs defined over a number
field are defined over $\oQ$ if and only if it holds true for special points.
  
\end{abstract}

\maketitle

\section{Introduction} \label{intro}

\subsection{Hodge loci}

The main object of study in this article are Hodge loci. Let us start
by recalling their definition in the geometric case, where their
behaviour is predicted by the Hodge conjecture.

\subsubsection{The geometric motivation}
Let $f:X \to S$ be a smooth projective morphism of smooth irreducible  complex
quasi-projective varieties and let $k$ a positive integer. The Betti
and De Rham incarnation of the $2k$-th cohomology of the fibers of $f$
give rise to a weight zero polarizable variation of Hodge structure $(\VV:=R^{2k}f^\an_* \ZZ (k),
\cV:= R^{2k}f_* \Omega^\bullet_{X/S}, F^\bullet, \nabla)$ on
$S$. Here $\VV$ is the local system on the complex manifold $S^\an$
associated to $S$ parametrizing the $2k$-th Betti cohomology of the
fibers of $f$; $\cV$ is the corresponding algebraic vector bundle,
endowed with its flat Gau\ss\--Manin
connection; and $F^\bullet$ is the Hodge
filtration on $\cV$ induced by the stupid
filtration on the algebraic De Rham complex 
$\Omega^\bullet_{X/S}$. In this situation one defines the locus of
exceptional Hodge classes $\Hod(\cV) \subset \cV^\an$ as the set of Hodge classes $\lambda \in F^0\cV^\an \cap
\VV_\QQ$ whose orbit under monodromy is infinite, and the Hodge locus $\HL(S,
\VV)$ as its projection in $S^\an$. Thus $\HL(S, \VV)$
is the subset of points $s$ in $S^\an$ for which the Hodge structure
$H^{2k}(X_s, \ZZ(k))$ admits more Hodge classes
than the very general fiber $H^{2k}(X_{s'}, \ZZ(k))$.

According to the Hodge conjecture each $\lambda \in \Hod(\cV)$ should be the cycle
class of an exceptional algebraic cycle in the corresponding fiber of
$f$. As algebraic subvarieties of the fibers are parametrized by a
common Hilbert scheme, the Hodge conjecture and an easy countability
argument implies the following (as noticed by 
Weil in \cite{Weil}, where he asks for an unconditional proof):

\medskip
\begin{tabular}{l|m{13cm}}
 $(\star) $ &The locus of Hodge classes $\Hod(\cV)$ is a countable union of closed irreducible
algebraic subvarieties of $\cV$. The restriction of $f$ to any such subvariety of $\cV$ is finite over its image. In particular
the Hodge locus
$\HL(S, \VV)$ is a countable union of closed irreducible
              algebraic subvarieties of $S$.
\end{tabular}

\subsubsection{Algebraicity of Hodge loci}
More generally let $(\VV, \cV, F^\bullet, \nabla)$ be any
polarizable variation of $\ZZ$-Hodge
structure ($\ZZ$VHS) on a smooth complex irreducible algebraic 
variety $S$. Thus $\VV$ is a finite rank $\ZZ_{S^\an}$-local system on
the complex manifold $S^\an$; and $(\cV, F^\bullet,
\nabla)$ is the unique regular algebraic module with integrable connection on
$S$ whose analytification is $\VV \otimes_{\ZZ_{S^\an}} \cO_{S^{\an}}$ endowed
with its Hodge filtration $F^\bullet$ and the holomorphic flat connection
$\nabla^\an$ defined by $\VV$, see
\cite[(4.13)]{Schmid}). We will abbreviate the $\ZZ$VHS
$(\VV, \cV, F^\bullet, \nabla)$ simply by $\VV$. 
If we define the locus of exceptional Hodge classes $\Hod(\cV) \subset \cV$ and the Hodge locus $\HL(S,
\VV) \subset S$ as in the geometric case, Cattani, Deligne and Kaplan
\cite{CDK95} proved a vast generalization of Weil's expectation:
\begin{theor}(Cattani-Deligne-Kaplan) \label{CDK} Let $\VV$ be a $\ZZ$VHS on a
  smooth complex quasi-projective variety
  $S$. Then $(\star)$ holds true.
\end{theor}

From now on we do not distinguish a complex algebraic variety $X$ from its associated complex analytic
space $X^\an$, the meaning being clear from the context.
It will be convenient for us to work in the following more general tensorial setting.
Let $\VV^\otimes$ be the  infinite direct sum of $\ZZ$VHS $\bigoplus_{a, b
  \in  \NN} \VV^{\otimes a} \otimes (\VV^\vee)^{\otimes b}$, where
$\VV^\vee$ denotes the $\ZZ$VHS dual to $\VV$; and let $(\cV^\otimes,
F^\bullet)$ be the corresponding filtered algebraic vector bundle of
infinite rank. We denote by $\Hod(\cV^\otimes) \subset \cV^\otimes$ and $\HL(S,
\VV^\otimes) \subset S$ the corresponding locus of Hodge tensors and
the tensorial Hodge locus respectively. Thus $\HL(S, \VV^\otimes)$ is the subset of
points $s$ in $S^\an$ for which the Hodge structure $\VV_s$ admits
more Hodge {\em tensors}
than the very general fiber $\VV_{s'}$. \Cref{CDK} says that 
$\Hod(\cV^\otimes)$ and $\HL(S, \VV^\otimes)$ are countable unions of
closed irreducible subvarieties of $\cV^\otimes$ and $S$ respectively,
called {\em the special subvarieties of $\cV^\otimes$ and $S$ for $\VV$}.  We
refer to \cite{BKT} for a simplified proof of the statement for
$\HL(S, \VV^\otimes)$ using o-minimal geometry.

\subsection{Fields of definition of Hodge loci}

The question we attack in this paper is the relation between the field
of definition of the $\ZZ$VHS $\VV$ and the fields of definition of
the corresponding special subvarieties. 

\subsubsection{The geometric case}
Once again the geometric case again provides us with a motivation and
a heuristic. Suppose that $f:X \to S$ is defined
over a number field $L\subset \CC$. In that case one easily checks,
refining Weil's argument, that the Hodge conjecture implies, in addition to $(\star)$:

\medskip
\begin{tabular}{l|m{13cm}}
$(\star \star) $ & (a) each irreducible component of $\Hod(\cV)$,
                   respectively $\HL(S, \VV)$, is defined over a finite extension of $L$.

                   (b) each of the finitely many $\Gal(\oQ/L)$-conjugates of such a
component is again an irreducible component of $\Hod(\cV)$,
respectively $\HL(S, \VV)$.
\end{tabular}

\begin{rem}
Of course $(\star \star)$ for $\Hod(\cV)$ implies $(\star \star)$ for
$\HL(S, \VV)$, and is a priori strictly stronger.
\end{rem}

\begin{rem}
The full Hodge conjecture is not needed to expect $(\star \star)$ to hold.
As proven by Voisin \cite[Lemma 1.4]{V07}, the property $(\star \star)$ for $\Hod(\cV)$ is equivalent
to the conjecture that Hodge classes in the fibers of $f$ are
(de Rham) absolute Hodge classes. We won't use the notion of
absolute Hodge classes in this article and refer the interested reader
to \cite{CS} for a survey.
\end{rem}

\subsubsection{Variations of Hodge structure defined over a number field}

Let us now turn to general $\ZZ$VHS.

\begin{defi}
We say that a $\ZZ$VHS $\VV$ is {\em defined over a number field $L\subset \CC$} if 
$S$, $\cV$, $F^\bullet$ and $\nabla$ are defined over 
$L$: $S= S_K \otimes_K \CC$,
$\cV = \cV_K \otimes_K \CC$, $F^\bullet\cV = (F_{K}^\bullet
\cV_K) \otimes_K \CC$ and $\nabla = \nabla_K \otimes_K \CC$
with the obvious compatibilities.

\end{defi}

In the same way the property $(\star)$, which is
implied by the Hodge conjecture in the geometric case, was proven to
be true for a general $\ZZ$VHS, we expect the property $(\star \star)$, which is
implied by the Hodge conjecture in the geometric case, to hold true
for any $\ZZ$VHS $\VV$, namely:

\begin{conj} \label{conj1}
  Let $\VV$ be a $\ZZ$VHS defined over a number field $L \subset
  \CC$. Then:
  \begin{enumerate}
\item[(a)]any special subvariety of $\cV^\otimes$, resp. $S$, for $\VV$ is
defined over a finite extension of $L$. 
\item[(b)] any of the finitely many $\Gal(\oQ/L)$-conjugates of a
  special subvariety of $\cV^\otimes$, resp. $S$, for $\VV$ is a
  special subvariety of $\cV^\otimes$, resp. $S$, for $\VV$.
  \end{enumerate}
\end{conj}

\begin{rem}
  Simpson conjectures that any $\ZZ$VHS defined
over a number field $L \subset \CC$ ought to be motivic: 
there should exist a $\oQ$-Zariski-open subset $U \subset S$ such that
the restriction of $\VV$ to $U$ is a direct factor of a geometric $\ZZ$VHS
on $U$, see \cite[``Standard conjecture'' p.372]{Si90}. Thus \Cref{conj1} would follow from Simpson's ``standard
conjecture'' and $(\star \star)$ in the geometric case. Of course
Simpson's standard conjecture seems unreachable with current techniques.
\end{rem}

\medskip
Let us mention the few results in the direction of \Cref{conj1} we are aware of:

\medskip
Suppose we are in the geometric situation of a morphism $f: X \to
S$ defined over $\QQ$.  In \cite[Theor. 0.6]{V07} (see also
\cite[Theor. 7.8]{Voisin2}), Voisin proves the following:

(1) for $\Hdg(\cV)$: let $Z \subset \cV$ is an irreducible component of
$\Hod(\cV)$ through a Hodge class $\alpha \in 
H^{2k}(X_0, \ZZ(k))_\prim$ such that the only constant sub-$\QQ$VHS
of the base change of $\VV_\QQ$ to $Z$ is $\QQ\cdot \alpha$. Then $Z$ is
defined over $\oQ$.

(2) for $\HL(S, \VV)$: under the weaker assumption that any
constant sub-$\QQ$VHS of the base change of $\VV_\QQ$ to $Z$ is purely
of type $(0,0)$, the irreducible component $p(Z)$ of $\HL(S, \VV)$ is defined over $\oQ$ and its 
$\Gal(\oQ/\QQ)$-translates are still special subvarieties of $S$ for
$\VV^\otimes$.

\medskip
\noi
In the case of a general $\ZZ$VHS Saito and Schnell \cite{SS}
prove:

(1) for $\Hod(\cV)$: if $\VV$ is defined over a number field then a
special subvariety of $\cV$ for $\VV$ is defined over $\oQ$ 
it contains a $\oQ$-point of $\cV$.

(2) for $\HL(S, \VV)$: without assuming that $\cV$ is defined over $\oQ$
but only assuming that $S$ is defined over a number field $L$, then a
special subvariety of $S$ for $\VV$ is defined over a finite extension
of $L$ if and only if it contains a $\oQ$-point of $S$.
This generalizes the well-known fact
that the special subvarieties of Shimura varieties are defined over $\oQ$
(as any special subvariety of a Shimura variety contains a CM-point,
and CM-points are defined over $\oQ$).

\begin{rem}
These results seem to indicate a significant gap in difficulty between \Cref{conj1} for
$\Hod(\cV)$ and \Cref{conj1} for $\HL(S, \VV)$. Saito and Schnell's
result (2), which only requires $S$ to be defined over $\oQ$, looks
particularly surprising. 
They also seem to indicate that the statement (b) in
\Cref{conj1} goes deeper than (a). In particular Saito and Schnell's
result (2) says nothing about Galois conjugates.
\end{rem}

\begin{rem}
Voisin's and Saito-Schnell's criteria look difficult
to check in practice. Even in explicit examples one usually knows very little about the
geometry of a special variety $Y$. In Voisin's case one would need to control the Hodge structure on
the cohomology of a smooth compactification of $X$ base-changed to
$Z$. In Saito-Schnell's case there is in general no natural
source of $\oQ$-points (like the CM points in the Shimura case).
\end{rem}

\subsection{Main results}
All results in this paper concern \Cref{conj1} for $\HL(S, \VV)$. We
provide a simple geometric criterion for a special 
subvariety of $S$ for $\VV$ to be defined over $\oQ$ and its Galois
conjugates to be special. 

\medskip
Let us first recall the notion of algebraic monodromy group.

\begin{defi} \label{mono}
Let $S$ be a smooth irreducible complex algebraic variety, let $k$ be a field and
$\VV$ a $k$-local system (of finite rank) on $S^\an$ (in our case $k$
will be $\QQ$ or $\CC$). Given an irreducible closed
subvariety $Y \subset S$, the algebraic monodromy group $\HH_Y$ of $Y$
for $\VV$ is the $k$-algebraic group connected component of the Tannaka group 
of the category $\langle \VV_{| Y^\norm} \rangle_{k\textnormal{Loc}}^\otimes$ of $k$-local systems on (the
normalisation of) $Y$ tensorially generated by the restriction of
$\VV$ and its dual.
\end{defi}
  
Equivalently $\HH_Y$ is the connected component of the Zariski-closure of the
monodromy $\rho: \pi_1(Y^{\norm,\an}) \to \GL(V_k)$ of the local system
$\VV_{| Y^\norm}$.

\begin{defi} \label{non-factor}
Let $S$ be a smooth irreducible complex algebraic variety and
$\VV$ a $k$-local system on $S^\an$. 
Let $Y \subset S$ be an irreducible closed subvariety.
We say that $Y$ is {\em weakly non-factor  for $\VV$} if it is not contained
in a closed irreducible $Z\subset S$ such that the $k$-algebraic monodromy
group $\HH_Y$ is a strict normal
subgroup of $\HH_Z$. We say that $Y$ is {\em positive dimensional  for
  $\VV$} if $\HH_Y \not = \{1\}$.
\end{defi}

\begin{rem} \label{extension of scalars1}
If $\VV$ is a $k$-local system on $S$, $Y \subset S$ is a closed
irreducible subvariety, and $k'$ is a field extension of $k$, the
$k'$-algebraic monodromy group $\HH_{Y}(\VV \otimes_k k')$ is the base
change $\HH_Y(\VV) \otimes_k k'$. Thus being weakly non-factor for
$\VV$ and  positive dimensional  for
  $\VV$ is equivalent to being weakly non-factor for
$\VV \otimes_k k'$ and  positive dimensional  for
$\VV \otimes_k k'$ respectively.
\end{rem}

Our main result in this paper is the following:
\begin{theor} \label{main}
Let $\VV$ be a polarized variation of $\ZZ$-Hodge structure on a smooth
quasi-projective variety $S$. 
\begin{itemize}
  \item[(a)] if $S$ is defined over a number field $L$ then any
    special subvariety of $S$ for $\VV$ which is weakly non-factor for $\VV_\QQ$
    is defined over a finite extension of $L$;
    \item[(b)] if moreover $\VV$ is defined over $L$ then the finitely many
      $\Gal(\oQ/L)$-translates of such a special subvariety are also
      special, weakly non-factor subvarieties of $S$ for $\VV$.
      \end{itemize}
\end{theor}

As a first corollary we obtain \Cref{conj1} for maximal strict
special subvarieties of $S$ under a simplicity assumption on the generic
Mumford-Tate group:

\begin{cor} \label{cor1}
Let $\VV$ be a polarized variation of $\ZZ$-Hodge structure on a smooth
quasi-projective variety $S$, whose adjoint generic Mumford-Tate group
$\G_S^\ad$ is simple. Then:
\begin{itemize}
  \item[(a)] if $S$ is defined over a number field $L$ then 
any strict special subvariety $Y \subset S$ for $\VV$, which is positive
dimensional for $\VV$ and maximal for
these properties, is defined over $\oQ$.

\item[(b)] if $\VV$ is moreover defined over $L$ then the finitely many
      $\Gal(\oQ/L)$-translates of such a special subvariety  are special
      subvarieties of $S$ for $\VV$.
      \end{itemize}
\end{cor}

\Cref{main} also enables to reduce the \Cref{conj1}(a) for $\HL(S,
\VV)$ to the case of points:
\begin{cor} \label{cor2} ~
Special subvarieties for $\ZZ$VHSs defined over $\oQ$ are
defined over $\oQ$ if and only if it holds true for special
points.
\end{cor}

\section{$\ZZ$VHS versus local systems, Mumford-Tate group versus
  monodromy, special versus weakly special}

In this section we recall the geometric background providing the
intuition for \Cref{main}, namely the geometry of special subvarieties and
their generalization, the weakly special subvarieties. We refer to \cite{klin} and \cite{KO} for details. 

\medskip
Let $\VV_\QQ$ be a $\QQ$-local system on $S$ and $Y \subset S$ an irreducible
closed subvariety. In \Cref{mono} we recalled
the definition of the algebraic monodromy group $\HH_Y$ for
$\VV_\QQ$. Suppose now that $\VV_\QQ$ underlies 
a $\ZZ$VHS $\VV$ over $S$. In addition to $\HH_Y$, which depends only on the
underlying local system, one attaches a more subtle
invariant to $Y$ and $\VV$: the generic Mumford-Tate group $\G_Y$
i.e. the Tannaka group of the category $\langle 
\VV_{| Y^\norm} \rangle_{\QQ \textnormal{VHS}}^\otimes$ of $\QQ$VHS on the
normalisation of $Y$ tensorially generated by the restriction of $\VV$
and its dual. This group is usually much harder to compute than $\HH_Y$ as its definition is
not purely geometric. The $\ZZ$VHS $\VV$ is completely described by
its complex analytic period map
$\Phi_S: S^\an  \to X_S:= \Gamma
\backslash \cD_S $.
Here $\cD_S$ denotes the Mumford-Tate domain associated to the
generic Mumford-Tate group $\G_S$ of
$(S, \VV)$, $\Gamma_S \subset \G_S(\QQ)$ is an arithmetic
lattice and the complex analytic quotient $X_S$ is called the Hodge
variety associated to $\VV$. The special subvarieties of the Hodge
variety $X_S$ and their generalisation, the weakly special
subvarieties of $X_S$ are defined purely in group-theoretic terms,
see \cite[Def. 3.1]{KO}. One proves that the special subvarieties of
$S$ for $\VV$ are precisely the
irreducible components of the $\Phi_S$-preimage of the special
subvarieties of $X_S$, thus obtaining the following characterization,
see \cite[Def. 1.2]{KO}. 

\begin{prop} \label{special}
  Let $\VV$ be a $\ZZ$VHS on $S$. 
A {\em special subvariety} of $S$ for $\VV$ is a closed irreducible
algebraic subvariety $Y \subset S$ maximal among the closed
irreducible algebraic subvarieties of $S$ with generic
Mumford-Tate group $\G_Y$.
\end{prop}

Similarly, one defines a generalisation of the special subvarieties of
$X_S$, the so-called {\em weakly special} subvarieties of $X_S$, purely
in group-theoretic terms see \cite[Def. 3.1]{KO}. The weakly special
subvarieties of $S$ for $\VV$, which generalize the special ones, are
defined as the irreducible
components of the $\Phi_S$-preimage of the weakly special
subvarieties of $X_S$. Again one obtains the following characterization,
see \cite[Cor. 3.14]{KO}:

\begin{prop} \label{weakly special}
  Let $\VV$ be a $\ZZ$VHS on $S$.
A {\em weakly special subvariety} $Y\subset S$ for $\VV$ is a closed
irreducible algebraic subvariety $Y$ of $S$ maximal among the closed
irreducible algebraic subvarieties of $S$ with algebraic monodromy
group $\HH_Y$.
\end{prop}

A posteriori \Cref{weakly special} offers an alternative definition of the weakly special subvarieties of
$S$ for a $\ZZ$VHS $\VV$. It is important for us to notice that this alternative definition of the
weakly special subvarieties of $S$ for $\VV$ makes sense for $\VV$ any
$k$-local system on $S^\an$, $k$ a field:

\begin{defi} \label{ws}
Let $k$ be a field and let $\VV$ be a $k$-local system on $S$. We
define a {\em weakly special subvariety} $Y\subset S$ for $\VV$ to be a closed 
irreducible algebraic subvariety $Y$ of $S$ maximal among the closed
irreducible algebraic subvarieties of $S$ with algebraic monodromy
group $\HH_Y$.
\end{defi}

\begin{rem} \label{extension of scalars2}
Following \Cref{extension of scalars1} $Y$ being weakly special for
$\VV$ is equivalent to $Y$ being weakly special for $\VV \otimes_k
k'$.
\end{rem}
  
\medskip
For $\VV$ a $\ZZ$VHS and $Y \subset S$ an
irreducible closed subvariety there exists a unique weakly special
subvariety $\langle Y\rangle_{\ws}$ with algebraic monodromy group
$\HH_Y$ and a unique special subvariety
$\langle Y\rangle_{\textnormal{s}}$ with generic
Mumford-Tate group $\G_Y$ containing $Y$, see \cite[2.1.4]{KO}:
$$ Y \subset \langle Y\rangle_{\ws} \subset \langle
Y\rangle_{\textnormal{s}} \subset S\;\;.$$
When $\VV$ is a mere local system there exists by definition a weakly
special subvariety with algebraic monodromy group $\HH_Y$ and containing
$Y$ but its uniqueness is not clear to us.

\medskip
Let us now recall that for $\VV$ a $\ZZ$VHS special subvarieties of $S$ for $\VV$ can be
thought of as families of weakly special subvarieties. Indeed let $Y\subset S$ be a weakly special subvariety. 
A fundamental result of Deligne-Andr\'e \cite[Theor.1]{An92} states
that the group $\HH_Y$ is normal in (the derived group 
of) $\G_Y$. Following
\cite[Prop. 2.13]{KO}, the decomposition $\G_Y^\ad = \HH_Y^\ad \times {\G'}_Y^\ad$
induces a product decomposition $X_Y = wX_Y \times X'_Y$, where $X_Y$
is the smallest special subvariety of $X_S$ containing $\Phi_S(Y)$ and
$Y$ is (an irreducible component of) $\Phi_S^{-1}(wX_Y \times \{x'_0\})$ for
a certain point $x' \in X'_Y$ and a weakly special subvariety $wX_Y$
of $X_S$. All the (irreducible components of) the
preimages $\Phi_S^{-1}(wX_Y \times \{x'\})$, $x' \in X'_Y$, are weakly
special subvarieties of $S$ for $\VV$ that can be thought as Hodge
theoretic deformations of $Y$. In particular, there are only countably many special
subvarieties of $S$ for $\VV$, while there are uncountably many weakly
special ones, organized in countably many ``product families''.

\medskip
We can now make a few remarks on the notion of {\em weakly non-factor}
subvarieties defined in \Cref{non-factor}:
\begin{enumerate}[leftmargin=*]

  \item For $\VV$ a local system a closed irreducible subvariety
        $Y\subset S$ is weakly non-factor if and only if any weakly
        special subvariety $Y \subset Z \subset S$ with $\HH_Z =
        \HH_Y$ is weakly non-factor. When $\VV$ is a $\ZZ$VHS it
        amounts to saying that the  weakly
        special closure $\langle Y\rangle_{\ws} \subset S$ is weakly non-factor.
        
        \item Let $\VV$ be a $\ZZ$VHS. Given a closed irreducible subvariety $Y
          \subset S$, let $wX_Y \subset X_S$ be the smallest
          weakly special subvariety containing $\Phi_S(Y)$. It follows
          from the above description of the weakly special
          subvarieties that $Y$ is weakly
          non-factor for $\VV$ if and only if there does not exist $Y \subset Z
          \subset S$, with $Z$ closed 
irreducible, such that $wX_Z= wX_Y \times wX' \subset X_S$ with $wX'$
a positive dimensional weakly special subvariety of $X_S$. The ``weakly  non-factor'' condition is thus a
Hodge theoretic rigidity of $Y$. In particular one obtains the
following:
 
\begin{lem} \label{weakly nf is special}
  Let $\VV$ be a $\ZZ$VHS on $S$. Any weakly non-factor, weakly special subvariety of
  $S$ is special.
  \end{lem} 

\item The terminology ``weakly non-factor'' generalizes the
      terminology ``non-factor'' introduced by Ullmo \cite{Ullmo} for
      special subvarieties of Shimura varieties.

\item For $\VV$ a non-isotrivial local system on $S$, it follows from the
        definition that for any weakly non-factor
        subvariety $Y \subset S$  the algebraic monodromy group
        $\HH_Y$ is non-trivial. When $\VV$ is moreover a $\ZZ$VHS this
        last condition is equivalent to saying that $Y$ is {\em positive
  dimensional for $\VV$} in the sense of \cite{KO}: its image $\Phi_S(Y)$
 is not a point. 
\end{enumerate}

Given $S$ a smooth complex quasi-projective variety and $\VV$ a
complex local system,  we say that $\VV$ is defined over a number
field $L\subset \CC$ if both $S$ and
the algebraic module with integrable connection $(\cV, \nabla)$
corresponding to $\VV$ under the Deligne-Riemann-Hilbert correspondence
(see (\ref{Deligne}) below) are defined over $L$. \Cref{main} then follows
immediately from \Cref{weakly nf is special} and the
general result on local systems:

\begin{theor} \label{main1}
Let $S$ be a smooth complex quasi-projective variety and $\VV$ a
complex local system on $S^\an$. 
\begin{itemize}
  \item[(a)] Suppose that $S$ is defined over a number field $L$. Then
    any weakly special, weakly non-factor subvariety of $S$ for $\VV$  
is defined over a finite extension of $L$;
    \item[(b)] if moreover $\VV$ is defined over $L$, then any
      $\Gal(\oQ/L)$-translates of a weakly special,
      resp. weakly non-factor, subvariety of $S$ for
      $\VV$ is a weakly special, resp. weakly non-factor,
      subvariety of $S$ for $\VV$. 
      \end{itemize}
\end{theor}

\section{Proof of the main results}

\subsection{Proof of \Cref{main1}(b)}

\begin{proof}[\unskip\nopunct]

Let $S$ be a smooth complex quasi-projective variety, $\Loc_\CC(S^\an)$ the category of complex
local systems of finite rank on $S^\an$, $\MIC(S^\an)$ the category of
holomorphic modules with integrable connection on $S^\an$ and
$\MIC_r(S)$ the category of algebraic modules with regular integrable
connection on $S$. Following Deligne \cite[Theor.5.9]{De70}, the
analytification functor $\MIC_r(S) \to \MIC(S^\an)$ is an equivalence
of tensor categories. Composed with the Riemann-Hilbert correspondence this
provide an equivalence of tensor categories
\begin{equation} \label{Deligne}
  \MIC_r(S) \stackrel{\tau}{\simeq} \Loc_\CC(S^\an)\;\;.
  \end{equation}

  \medskip
Let $\VV \in \Loc_\CC(S^\an)$.
Let $\sigma: \CC \to \CC$ be a field automorphism. Let $S^\sigma:= S
\times_{\CC, \sigma} \CC$ be the twist of
$\SS$ under $\sigma$. We denote by $\VV^\sigma \in \Loc_\CC((S^\sigma)^\an)$ the
image of $\VV$ under the composition of
equivalence of (Tannakian) categories
\begin{equation} \label{e2}
  \Loc_\CC(S^\an) \stackrel{\tau^{-1}}{\sim} \MIC_r(S)
\stackrel{\cdot \times_{\CC, \sigma} \CC}{\sim} \MIC_r(S^\sigma)
\stackrel{\tau}{\sim} \Loc_\CC((S^\sigma)^\an)\;\;.
\end{equation}

\Cref{main1}(b) then follows immediately from the following more general:

\begin{prop} \label{twist}
Let $S$ be a smooth complex quasi-projective variety and $\VV \in
\Loc_\CC(S^\an)$. Let $\sigma: \CC \to \CC$ be a field
automorphism. Let $Y \subset S$ be a closed irreducible subvariety with Galois twist
$Y^\sigma \subset S^\sigma$.
\begin{enumerate}
  \item[(1)] the complex algebraic monodromy group $\HH_Y$
of $Y$ with respect to $\VV$ is canonically isomorphic to the complex
algebraic monodromy group $\HH_{Y^\sigma}$ of $Y^\sigma$ with respect
to $\VV^\sigma$.
\item[(2)] $Y$ is weakly special for $\VV$ if and only if $Y^\sigma$ is
  weakly special for $\VV^\sigma$.
  \item[(3)] $Y$ is weakly non-factor for $\VV$ if and only if $Y^\sigma$
    is weakly non-factor for $\VV^\sigma$.
    \end{enumerate}
  \end{prop}

  \begin{proof}
    Let us first assume that $Y$ is smooth. In that case the
    equivalence of tensor categories~(\ref{e2}) $\Loc_\CC(Y^\an)
    \stackrel{\tau}{\simeq} \Loc_\CC((Y^\sigma)^\an)$ restricts to an
    equivalence of tensor categories
    $$ \langle  \VV_{|Y} \rangle^\otimes \stackrel{\tau}{\simeq} \langle
    \VV_{|Y}^\sigma \rangle^\otimes \;\;. $$
    Taking (the connected component of the identity of) their Tannaka
    groups we obtain a canonical isomorphism
    $$\HH_Y \simeq \HH_{Y^\sigma}\;\;,$$
    thus proving \Cref{twist}(1) in that case.

    When $Y$ is not smooth, we
    consider a desingularisation $Y^s \stackrel{p}{\to} Y^\nor
    \stackrel{\pi}{\to} Y$. Notice that $(Y^s)^\sigma$ is a
    desingularisation of $(Y^\norm)^\sigma = (Y^\sigma)^\norm$. Notice
    moreover that the algebraic monodromy groups of $(p\circ \pi)^*
    \VV_{|Y}$ and $\pi^*\VV_Y$ coincides, as $p_* : \pi_1(Y^s) \to
    \pi_1(Y^\norm)$ is surjective. Arguing as above for $Y^s$ and
    $(Y^s)^\sigma$ proves \Cref{twist}(1) in general.

    Suppose now that $Y \subset S$ is a closed irreducible subvariety. If $Y^\sigma$ is not weakly special for
    $\VV^\sigma$ there exists $Z \supset Y^\sigma$ a closed
    irreducible subvariety of $S^\sigma$ containing $Y^\sigma$
    strictly and such that $\HH_Z =\HH_{Y^\sigma}$. But then
    $Z^{{\sigma^{-1}}}$ is a closed irreducible subvariety of $S$
    containing $Y$ strictly, and such that $\HH_{Z^{\sigma^{-1}}}=
    \HH_Y$ by \Cref{twist}(1). It follows that $Y$ is not weakly
    special. This proves \Cref{twist}(2).

    The argument for \Cref{twist}(3) is similar. We are reduced to
    showing that for $S$ a smooth complex quasi-projective variety,
    $\VV \in \Loc_\CC(S^\an)$, $\sigma: \CC \to \CC$ a field
automorphism and $Y \subset S$ a closed irreducible subvariety with Galois twist
$Y^\sigma \subset S^\sigma$, then $\HH_{Y}$ is
normal in $\HH_{S}$ if and only if $\HH_{Y^\sigma}$ is normal in
$\HH_{S^\sigma}$. Consider the tannakian subcategory $\cT$ of 
$\Loc_\CC(S^\an)$ consisting of the local systems which are trivial in
restriction to $Y^\an$. Applying $\sigma$ we obtain that $\cT^\sigma$
is the tannakian subcategory of $\Loc_\CC((S^\sigma)^\an)$ of local
systems that are trivial on $(Y^\sigma)^\an$. But as a result of the
tannakian formalism the
Tannaka group of $\cT$, resp. $\cT^\sigma$, are the normal closures of
$\HH_Y$ and $\HH_{Y^{\sigma}}$ in $\HH_S$ and $\HH_{S^{\sigma}}$
respectively. Hence the result.

  \end{proof}

  \end{proof}

\subsection{Proof of \Cref{main} when $\VV$ is defined over a number
  field}

\begin{proof}[\unskip\nopunct]
Although this is not logically necessary, let us notice that 
\Cref{main} in the case where $\VV$ is defined over a number field $L$
follows from \Cref{main1}(b). Indeed when $\VV$ is a $\ZZ$VHS, weakly
special weakly non-factor subvarieties of $S$ for $\VV$ are special
subvarieties of $S$ for $\VV$ by \Cref{weakly nf is special}. Applying
\Cref{main1}(b), it follows that the $\Aut(\CC/L)$-translates of any
special, weakly non-factor, subvariety of $S$ for $\VV$ is
special (and weakly non-factor). But special subvarieties of $S$ for $\VV$ form a countable
set. It follows immediately that any special, weakly non-factor,
subvariety of $S$ for $\VV$ is defined over $\oQ$ (see for instance
\cite[Claim p.25]{Voisin2}).

\end{proof}

\subsection{Proof of \Cref{main1}(a)}
\begin{proof}[\unskip\nopunct]
 Let us now prove \Cref{main1}(a), hence finish the proof of \Cref{main}.
Let $S$ be a complex irreducible smooth quasi-projective variety and
$\VV$ a complex local system on $S^\an$. Suppose that $S$ is defined
over a number field $L \subset \CC$.
Let $Y \subset S$ be a weakly special subvariety of $S$ for $\VV$ which is
weakly non-factor. Let us show that $Y$ is defined over $\oQ$.

\medskip
Let $ Z \subset S$ be the $\oQ$-Zariski-closure of $Y$, i.e. the
smallest closed subvariety of $S$ defined over $\oQ$ and containing $Y$. Thus
$Z$ is irreducible.

\medskip
The subset $Z^0 \subset Z$ of smooth points is $\oQ$-Zariski-open
(meaning that $Z - Z^0$ is a closed subvariety of $Z$ defined over
$\oQ$) and dense. Notice that $Y \cap Z^0$ is Zariski-open in $Y$
(otherwise $Y$ would be contained in the closed subvariety $Z - Z^0$
defined over $\oQ$, in contradiction to the $\oQ$-Zariski-density of $Y$
in $Z$); moreover the fact that $Y \subset S$ is weakly special,
resp. weakly non-factor for $(S, \VV)$ implies that $Y^0:=Y \cap
Z^0$ is weakly special, resp. weakly non-factor for 
$(Z^0, \VV_{|Z^{0}})$. Replacing $Y \subset S$ by $Y^0 \subset
Z^0$ if necessary, we can without loss of generality assume that $Y$
is $\oQ$-Zariski-dense in $S$. We are reduced to proving that $Y=S$,
or equivalently that $\HH_Y = \HH_S$. This follows immediately
from the \Cref{monodromy} below, of independent interest.
\end{proof}

\begin{prop} \label{monodromy}
  Let $S$ be a smooth complex quasi-projective variety, $\VV$ a
  complex local system on $S^\an$ and let $Y\subset S$ be a closed irreducible
weakly non-factor subvariety for $\VV$. Suppose that $S$ is defined over $\oQ$
and that $Y$ is $\oQ$-Zariski-dense in $S$. Then $\HH_Y = \HH_S$.
\end{prop}

\begin{proof}
Let $\cY$ be ``the'' spread of $Y$ with respect to $S$. Let us recall
its definition.
Let $K \subset \CC$ be the minimal
field of definition of $Y$, see
\cite[Cor. 4.8.11]{Gro}. This is the smallest subfield $\oQ \subset
K \subset \CC$ such that $Y$ is defined over $K$: there exists a $K$-scheme of finite type $Y_K$
such that $Y = Y_K \otimes_K \CC$. Let us choose $R \subset K$ a finitely generated
$\oQ$-algebra whose field of fractions is $K$ and let $\cY_R$ be an
$R$-model of $Y_K = \cY_R \otimes_R K$. The morphism $\cY_R \to \Spec
R$ induces a morphism of complex varieties $\cY := \cY_R \otimes_\oQ \CC \to T:= \Spec (R
\otimes_{\oQ} \CC)$, defined over $\oQ$. Notice that the complex dimension of $T$ is the transcendence degree of $K$ over
$\oQ$. The natural closed immersion $Y_R \subset S \otimes_\oQ R$
makes $\cY$ a closed irreducible variety $$\cY \subset S \times_\CC T$$
defined over $\oQ$, with induced projections $p: \cY \to S$ and
$\pi: \cY \to T$, both defined over $\oQ$, such that $\cY_{t_{0}} := \pi^{-1}(x_0) \simeq Y$
where $t_0 \in T(\CC)$ is the closed point given by $R \subset K
\subset \CC$. By construction the morphism $p$ is dominant. The variety $\cY$ is called ``the'' spread of $Y$. It
depends on the choice of $R$ but different choices give rise to
birational varieties $\cY$s. Shrinking $\Spec R$ if necessary, we can assume
without loss of generality that $T$ is smooth.

\medskip
Let $\cY^0 \subset \cY$ be the $\oQ$-Zariski-open dense subset of
smooth points. As $p$ is dominant, the fact that $Y \subset
  S$ is weakly non-factor for $(S, \VV)$ implies that $Y^0:= \cY^0
\cap Y \subset \cY^0$ is weakly non-factor for $(\cY^0,
p^{-1}(\VV)_{|\cY^{0}})$. As $\HH_{Y^0} = \HH_Y$ and $\HH_{\cY^0}=
\HH_S$, to show that $\HH_Y = \HH_S$ we are reduced, replacing $S$ by
$\cY^0$ and $Y$ by $\cY^0
\cap Y$ if necessary, to the situation where there exists a morphism $\pi: S
\to T$ defined over $\oQ$ such that $Y = S_{t_{0}} \subset S$ and $Y$
is weakly non-factor for $(S, \VV)$.

\medskip
It follows from \cite[Theorem p.57]{GM} that there exist finite
Whitney stratifications $(S_l)$ of $S$ and $(T_l)_{l \leq d}$ of
$T$ by locally closed algebraic subsets $T_l$ of dimension $l$ ($d=\dim T$) such
that for each connected component $Z$ (a stratum) of $T_l$,  $\pi^{-1}(Z)^\an$ is a topological fibre bundle over $Z^\an$, and a
union of connected components of strata of $(S_j^\an)$, each mapped
submersively to $Z^\an$ (moreover, for all $t \in Z^\an$, there exists an open neigbourhood $U(t)$ in
$Z^\an$ and a stratum preserving homeomorphism $h: \pi^{-1}(U) \simeq
\pi^{-1}(t) \times U$ such that $\pi_{|\pi^{-1}(U)} = p_U \circ h$,
where $p_U$ denotes the projection to $U$). These Whitney stratifications can be chosen defined over
$\oQ$ (meaning that the closure of each stratum is defined over
$\oQ$): see \cite{Teissier}, \cite[3.1.9]{Ar}.

\medskip
It follows from the minimality of $K$ that $t_0$ belongs to the unique open
stratum $T_d$, $d= \dim T$. Without loss of generality we can and will
assume
from now on that $T= T_d$. In particular  $S^\an$ is a topological
fibre bundle over $T^\an$.

\medskip
If follows that the image of $\pi_1(Y^\an)$ in $\pi_1(S^\an)$ is a
normal subgroup. Hence $\HH_Y$ is a normal subgroup of $\HH_S$. As $Y
\subset S$ is weakly non-factor it follows that $\HH_Y = \HH_S$. 

\end{proof}

  \subsection{Proof of \Cref{cor1}}
  \begin{proof}[\unskip\nopunct]
    Let $S$, $\VV$ and $Y$ as in the statement of \Cref{cor1}. Let us
    show that $Y$ is weakly non-factor.
    Let $Z \subset S$ be a closed irreducible subvariety of
    $S$ containing $Y$ strictly, and such that $\HH_Y$ is is a strict
    normal subgroup of $\HH_Z$.  As the special closure $\langle
Z\rangle_{\textnormal{s}}$ of $Z$ is a special subvariety of $(S,
\VV)$ containing $Y$, it follows from the maximality of $Y$ that $\langle
Z\rangle_{\textnormal{s}}=
S$. As $\HH_Z$ is normal (see \cite[Theor.1]{An92}) in the
algebraic group $\G_Z^\der = \G_S^\der$ which is assumed to be
simple,  it follows that either $\HH_Z = \{1\}$ or $\HH_Z = \HH_S=
\G_S^\der$. As $\HH_Y$ is a strict normal subgroup of $\HH_Z$,
necessarily $\HH_Y=\{1\}$ (and $\HH_Z= \HH_S$). This is impossible as $Y$ is positive
dimensional for $\VV$. Hence such a $Z$ does not exist and $Y$ is
weakly non-factor. The conclusion then follows from \Cref{main}.

\end{proof}

\subsection{Proof of \Cref{cor2}}

\begin{proof}[\unskip\nopunct]
Let us suppose that the special points for $\ZZ$VHS's defined over $\oQ$
are defined over $\oQ$. Let $\VV \to S^\an$ be a $\ZZ$VHS defined over
$\oQ$ and let $Y$ be a special
subvariety of $S$ for $\VV$. Let us show that $Y$ is defined over $\oQ$.

\medskip
Suppose for the sake of contradiction that $Y$ is not defined over
$\oQ$. Let $Z\subset S$ be the $\oQ$-Zariski closure $Z$ of
$Y$ in $S$. Again, replacing $S$ by the $\oQ$-Zariski open subset of smooth
points $Z^0$ of $Z$ and $Y$ by $Y^0 := Z^0 \cap Y$ we can without loss of
generality assume that $Z=S$ is smooth. Arguing as in the proof of
\Cref{main1}(a) we may assume 
that $\HH_Y$ is a strict normal subgroup of $\HH_S$, hence of
$\G_S$.

\medskip
It follows that there exist a finite collection of
natural integers $a_i, b_i$, $1 \leq i \leq n$ such that the $\ZZ$VHS
$\VV':= (\bigoplus_{1\leq i \leq n}\VV^{\otimes a_i} \otimes (\VV^\vee)^{\otimes
  b_i})^{\HH_Y}$ consisting of the $\HH_Y$-invariant vectors in
$\bigoplus_{1\leq i \leq n}\VV^{\otimes a_i} \otimes
(\VV^\vee)^{\otimes b_i} $ has generic
Mumford-Tate group $\G'_S= \G_S /\HH_Y$ and algebraic monodromy group
$\HH'_S:= \HH_S/\HH_Y$. Writing $(\G'_S= \G_S /\HH_Y, \cD'_S:= \cD_S/\HH_Y)$ for the
quotient Hodge datum of $(\G_S, \cD_S)$ by $\HH_Y$ and $\pi: X_S \twoheadrightarrow
X'_S$ the induced projection of Hodge varieties, the period map for
$\VV'$ is $\Phi'_S:= \pi \circ \Phi_S: S^\an \to X'_S$. The special
subvariety $Y$ of $S$ for $\VV$ is still a special subvariety of $S$
for $\VV'$ and its image $\Phi'_S(Y)$ is a point.

\medskip
Following \cite[Theor.1.1]{BBT} there exists a
factorisation $$ \Phi'_S = \Psi \circ q\;\;,$$
where $q: S \to B$ is a proper morphism of quasi-projective varieties
defined over $\oQ$ and $\Psi: B \to X'$ is a quasi-finite period
map. This means that $\VV' = q^* \VV'_B$ for a $\ZZ$VHS
$\VV'_B$, and that $b_0:= q(Y)$ is a special point of $B$ for
$\VV'_B$.

\medskip
It follows from \Cref{subvariation} below that the $\ZZ$VHS 
$\VV'$ can be defined over $\oQ$. It then follows from \Cref{descent}
below that $\VV'_B$ is also defined over $\oQ$. Under our assumption
that special points of $\ZZ$VHS defined over $\oQ$ are defined over
$\oQ$ one concludes that the special point $b_0$ of $B$ for $\VV'_B$
is defined over $\oQ$. But then the irreducible component $Y$ of
$q^{-1}(b_0)$ is also defined over $\oQ$, a contradiction.

\medskip
This finishes the proof of \Cref{cor2}.
 
\end{proof}

\begin{lem} \label{subvariation}
  Let $\VV$ be a $\ZZ$VHS and $\VV'$ a sub-$\ZZ$VHS. If $\VV$ is
  definable over $\oQ$ then there exists a $\oQ$-structure on $\VV$
  and $\VV'$ such that the projection $\VV\twoheadrightarrow \VV'$ is defined over $\oQ$.
\end{lem}

\begin{proof}
  Let $E$ be the finite dimensional $\oQ$-algebra of $\nabla$-flat 
  $F^\bullet$-preserving algebraic sections over $S$ of $\cV_\oQ \otimes
  \cV_\oQ^\vee$. Each invertible element of $E_\CC:=E
  \otimes_\oQ \CC$ defines a natural $\oQ$-structure on $\cV$,
  $F^\bullet$ and $\nabla$, the
  original one $(\cV_\oQ, \F^\bullet_\oQ, \nabla_\oQ)$ being preserved exactly by the invertible elements of
  $E$.

  \medskip
  Let $J$ be the Jacobson radical of $E$. Let us choose $T
  \subset E$ a (semi-simple) splitting of the projection $E \to
  E/J$. As the category of polarizable $\QQ$VHS is abelian semi-simple
  the finite dimensional complex algebra
  $\Hom_{\ZZ\textnormal{VHS}}(\VV, \VV)\otimes_\ZZ \CC$ is semi-simple.
  Under the Riemann-Hilbert correspondence it identifies with a
  semi-simple subalgebra $\cA \subset E_\CC$. Following a classical
  result of Wedderburn-Malcev there exists an element $j \in  J_\CC:=
  J \otimes_\oQ \CC$ such that $(1+j) \cA (1+j)^{-1} \subset T_\CC$.

  \medskip
 Let $e_\CC \in \cA$ be the idempotent of $S_\CC$ corresponding to the
 projection of $\ZZ$VHS $\pi: \VV \twoheadrightarrow\VV'$ under the
 Riemann-Hilbert correspondence. As $T_\CC$ is
 semi-simple, hence a product of matrix algebras, any idempotent of $T_\CC$ is conjugated to an idempotent
 in $S$. Thus there exist an invertible element $f \in T_\CC$ and $e \in T$ such that $(1+j)
 e_\CC (1+j)^{-1}= f^{-1} e f$.

 \medskip
 If we endow $(\cV, F^\bullet, \nabla)$ with the $\oQ$-structure
 defined by the element $f (1+j) \in E_\CC$ it follows that the image
 of $\pi: \VV \twoheadrightarrow \VV'$ under the Riemann-Hilbert correspondence
 is defined over $\oQ$ for this new $\oQ$-structure.  Hence the result.

\end{proof}

\begin{lem} \label{descent}
  Let $f: S \lo B$ be a proper morphism of $\oQ$-varieties defined
  over $\oQ$, such that $f_* \cO_{S} = \cO_{B}$. Let $\VV_B$
  be a $\ZZ$VHS on $B$. If the $\ZZ$VHS $\VV_S:=f^*\VV_B$ on $S$ is definable over
  $\oQ$ then $\VV_B$ is also definable over $\oQ$.
\end{lem}

\begin{proof}
 Let $(\cV_S:= f^* \cV_B, F^\bullet_S:= f^* F_B^\bullet, \nabla_S:=
 f^* \nabla_B)$ be the De Rham incarnation of $\VV_S$. 
  It follows from the projection formula and the assumption $f_* \cO_{S} = \cO_{B}$ that
  $$ f_* \cV_S = f_*( f^* \cV_B
  \otimes_{\cO_S} \cO_S) = \cV_B \otimes_{\cO_B} f_* \cO_S
  = \cV_B \;\;.$$
  It follows easily that $F^\bullet_B = f_* F^\bullet_S$ and
  $\nabla_B= f_*
  \nabla_S$. As $f$, $F_S^\bullet$ and $\nabla_S$ are defined over
  $\oQ$, it follows that  $F^\bullet_B$ and $\nabla_B$ are defined
  over $\oQ$.

\end{proof}

\begin{rem}
The companion statement to \Cref{cor2} that conjugates of special varieties for
$\ZZ$VHSs defined over a number field are special if
and only if it holds true for special points would follow from a
version of \Cref{subvariation} over a fixed number field $L$ rather than
over $\oQ$, but this last version is not clear to us.
 \end{rem}

\medskip
\noindent Bruno Klingler : Humboldt Universit\"at zu Berlin

\noindent email : \texttt{bruno.klingler@hu-berlin.de}

\medskip
\noindent Ania Otwinowska: Humboldt Universit\"at zu Berlin

\noindent email : \texttt{anna.otwinowska@hu-berlin.de}

\medskip
\noindent David Urbanik: University of Toronto

\noindent email : \texttt{david.b.urbanik@gmail.com}

\begin{thebibliography}{99}

\bibitem[An92]{An92} Y. Andr\'e, {\it Mumford-Tate groups of mixed Hodge  
structures and the theorem of the fixed part},
Compositio Math. {\bf 82} (1992) 1-24

\bibitem[Ar13]{Ar} D. Arapura, {\it An abelian category of motivic
      sheaves}, Adv. Math. {\bf 233} (2013), 135-195

\bibitem[BBT18]{BBT} B. Bakker, Y. Brunebarbe, J. Tsimerman, {\it
      o-minimal GAGA and a conjecture of Griffiths}, \url{http://arxiv.org/abs/1811.12230}

    \bibitem[BKT18]{BKT} B. Bakker, B. Klingler, J. Tsimerman, {\it Tame
    topology of arithmetic quotients and algebraicity of Hodge loci},
  \url{http://arxiv.org/abs/1810.04801}, to appear in JAMS
  
\bibitem[CDK95]{CDK95} E. Cattani, P. Deligne, A. Kaplan, {\it On the locus of Hodge classes.}
J. of AMS, {\bf 8} (1995), 483-506

\bibitem[ChSc14]{CS} F. Charles, C. Schnell, {\it Notes on absolute Hodge
  classes} , in Hodge Theory, Math. Notes {\bf 49} (2014), 469-530

\bibitem[De70]{De70} P. Deligne, {\it Equations diff\'erentielles \`a
    points singuliers r\'eguliers}, LNM {\bf 163} (1970) 

\bibitem[De72]{De72}  P. Deligne, {\it La conjecture de Weil pour les surfaces
    $K3$}, Invent. Math. {\bf 15} (1972), 206-226

\bibitem[GM88]{GM} M. Goresky, R. MacPherson, {\it Stratified Morse
    theory}, Ergebnisse der Mathematik und ihrer Grenzgebiete {\bf 14}
  (1988)
  
\bibitem[Gro65]{Gro} A. Grothendieck, {\it El\'ements de g\'eom\'etrie
    alg\'ebrique}, IV-2, Publ. Math. IHES {\bf 24} (1965)

\bibitem[K17]{klin} B. Klingler, Hodge theory and atypical intersections: conjectures,
 \url{http://arxiv.org/abs/1711.09387}, to appear in the book {\em Motives
    and complex multiplication}, Birkha\"user 

\bibitem[KO19]{KO} B. Klingler, A. Otwinowska, {\it On the closure of
    the positive Hodge locus}, \url{http://arxiv.org/abs/1901.10003}
  

 \bibitem[SaSc16]{SS} M. Saito, C. Schnell, {\it Fields of definition of
        Hodge loci}, in Recent advances in Hodge theory, LMS {\bf
        427} (2016) 275-291

      \bibitem[Sc73]{Schmid} W. Schmid, {\it Variation of Hodge structure:
    the singularities of the period mapping}, Invent. Math. {\bf 22}
  (1973), 211-319

  \bibitem[Si90]{Si90} C. Simpson, {\it Transcendental aspects of the
    Riemann-Hilbert correspondence}, Illinois J. Math. {\bf 34},
  no. 2, 368-391 (1990)

 \bibitem[Tei82]{Teissier} B. Teissier, {\it Vari\'et\'es polaires,
     II}, Lect. Notes in Math. {\bf 961} (1982), 314-491 

 \bibitem[Ull07]{Ullmo} E. Ullmo, {\it Equidistribution de
     sous-vari\'et\'es sp\'eciales. II}, J. Reine Angew. Math. {\bf
     606} 193-216 (2007)
  
\bibitem[Voi07]{V07} C. Voisin, {\it Hodge loci and absolute Hodge
      classes}, Compositio Math. {\bf 143}, 945-958, (2007)

\bibitem[Voi13]{Voisin2}  C. Voisin, {\it Hodge loci}, in Handbook of
  moduli, Vol. III, 507-546, Adv. Lect. Math. 26, Int. Press,
  Somerville, MA, 2013

\bibitem[Weil79]{Weil} A. Weil, {\it Abelian varieties and the Hodge
ring}, Collected papers III, Springer Verlag, 421-429 (1979)


\end{thebibliography}
\end{document}